\documentclass[a4paper,12pt]{amsart}
\usepackage[frenchstyle,largesmallcaps,fulloldstylenums,nott,easyscsl]{kpfonts}
\usepackage{xfrac}
\usepackage{mathtools}
\makeatletter
\def\@settitle{\baselineskip14\p@\relax
    \flushleft{\Large\bfseries \@title}
}

\def\@setauthors{%
  \begingroup
  \def\thanks{\protect\thanks@warning}%
  \trivlist
  \centering\footnotesize \@topsep30\p@\relax
  \advance\@topsep by -\baselineskip
  \item\relax
  \author@andify\authors
  \def\\{\protect\linebreak}%
  \flushleft{\normalfont by \authors\\}\relax\vskip2\baselineskip%
  \centerline{\itshape Uffe Haagerup (1949--2015) in Memoriam}%
  \ifx\@empty\contribs
  \else
    ,\penalty-3 \space \@setcontribs
    \@closetoccontribs
  \fi
  \endtrivlist
  \endgroup
  \let\shortauthors\authors
}

\def\maketitle{\par
  \@topnum\z@ 
  \@setcopyright
  \thispagestyle{firstpage}
  \ifx\@empty\shortauthors \let\shortauthors\shorttitle
  \else \andify\shortauthors
  \fi
  \@maketitle@hook
  \begingroup
  \@maketitle
  \author@andify\authors
  \toks@\@xp{\authors}\@temptokena\@xp{\shorttitle}%
  \toks4{\def\\{ \ignorespaces}}
  \edef\@tempa{%
    \@nx\markboth{\the\toks4
      \the\toks@}{\the\@temptokena}}%
  \@tempa
  \endgroup
  \c@footnote\z@
  \@cleartopmattertags
}

\def\ps@headings{\ps@empty
  \def\@evenhead{%
    \setTrue{runhead}%
    \normalfont\scriptsize
    \llap{\thepage\hskip1em}
    \def\thanks{\protect\thanks@warning}%
    \leftmark{}{}\hfil}%
  \def\@oddhead{%
    \setTrue{runhead}%
    \normalfont\scriptsize \hfil
    \def\thanks{\protect\thanks@warning}%
    \rightmark{}{}\rlap{\hskip1em\thepage}}%
  \let\@mkboth\markboth
}
\def\section{\@startsection{section}{1}{\z@}%
    {-21dd plus-4pt minus-4pt}{10.5dd plus 4pt
     minus4pt}{\normalsize\bfseries\boldmath}}

\ifx\ParFact\undefined\def\ParFact{0.4}\fi
\def\@setparskip{\parskip=\ParFact\baselineskip \advance\parskip by 0pt plus 2pt}
\@setparskip
\makeatother

\usepackage{tikz}
\usetikzlibrary{matrix}
\usepackage{caption}
\usepackage[flushmargin]{footmisc}
\usepackage{varwidth}
\usepackage{enumitem}
\setlist[enumerate]{labelindent=--.5in,leftmargin=0pt}
\setlist[itemize]{labelindent=--.5in,leftmargin=0pt}
\usepackage{nicefrac}
\usepackage{euscript}
\usepackage{color}
\definecolor{green}{RGB}{0,127,0}
\definecolor{red}{RGB}{191,0,0}
\usepackage[colorlinks=true]{hyperref}

\theoremstyle{plain}
\newtheorem{theorem}{Theorem}[section]
\newtheorem{corollary}[theorem]{Corollary}
\newtheorem{lemma}[theorem]{Lemma}
\newtheorem{proposition}[theorem]{Proposition}
\theoremstyle{definition}

\newtheorem{remark}[theorem]{Remark}
\newtheorem{question}[theorem]{Question}
\newtheorem{problem}[theorem]{Problem}

\numberwithin{equation}{section}

\newcommand{\B}[1]{\mathbf{#1}}
\newcommand{\C}[1]{\mathcal{#1}}
\newcommand{\F}[1]{\mathfrak{#1}}

\begin{document}

\title[Positive definite functions on Coxeter groups\dots]{Positive definite functions on Coxeter groups
with applications to operator spaces and noncommutative
probability}
\author{Marek {\sc Bo\.zejko}}
\address{M.~Bo{\.z}ejko --- \normalfont Polska Akademia Nauk, ul.~Kopernika 18, 50-001 Wroc{\l}aw}
\email{bozejko@math.uni.wroc.pl}
\author{\'Swiatos{\l}aw R.~{\sc Gal}}
\address{\'S.~R.~Gal --- \normalfont Uniwersytet  Wroc{\l}awski, pl.~Grunwaldzki 2/4, 48-300 Wroc{\l}aw}
\email{sgal@math.uni.wroc.pl}
\author{Wojciech {\sc M{\l}otkowski}}
\address{W.~M{\l}otkowski --- \normalfont Uniwersytet  Wroc{\l}awski, pl.~Grunwaldzki 2/4, 48-300 Wroc{\l}aw}
\email{mlotkow@math.uni.wroc.pl}

\begin{abstract}
A new class of positive definite functions
related to colour-length function on arbitrary Coxeter group is introduced.
Extensions of positive definite functions, called the Riesz-Coxeter product,
from the Riesz product on the Rademacher (Abelian Coxeter) group to arbitrary
Coxeter group is obtained. Applications to harmonic analysis, operator spaces
and noncommutative probability is presented. Characterization of radial and
colour-radial functions on dihedral groups and infinite permutation group are shown.
\end{abstract}

\maketitle

\def\cb{{\mathop{\rm cb}}}
\def\tr{\mathop{\rm tr}}

\section*{Introduction}

In 1979 Uffe Haagerup in his seminal paper \cite{MR520930}, essentially proved
the positive definitness, for $0\leq q\leq 1$, of the function  $P_q(x) =
q^{|x|} = \exp(-t|x|)$, where $|\cdot|$ is the word lenght on a free Coxeter
group $W=\B Z/2*\dots *\B Z/2$. From this he deduced also Khinchine
type inequalities. He has shown that the regular $C^*$-algebra of $W$ has
bounded approximation property and later \cite{MR784292} the completely bounded approximation
property \textsc{\MakeLowercase{(CBAP)}}. These results of Uffe Haagerup have had
significant impact on harmonic analysis on free groups and, more generally, on
Coxeter groups; they also influenced free probability theory and other
noncommutative probability theories.

In the paper \cite{MR950825} it was shown that the function $P_{q} (x) = q^{|x|}$
is positive definite for $q\in[-1,1]$ and  all Coxeter groups,
where the length $|\cdot|$ is the natural word length function on a Coxeter group
with repect to the set of its Coxeter generators. This fact implies that infinite
Coxeter groups have the Haagerup property and do not have Kazhdan's
propery (T).

Later, Januszkiewicz \cite{MR1925487} and Fendler \cite{Fendler} showed,
in the spirit of Haagerup proof, that $w\mapsto z^{|w|}$
is a coefficient of a uniformely bounded Hilbert representation of $W$
for all $z\in\B C$ such that $|z|< 1$.  As shown in a very short
paper of Valette \cite{MR1172955}, this implies \textsc{\MakeLowercase{CBAP}}.
See the book \cite{MR2391387} for
futher extension of Uffe Haagerup results for a big class of
groups.

In the paper \cite{MR1388006} Bo{\.z}ejko and Speicher considered the free
product (convolution) of classic normal distribution $N(0,1)$ and the new
length function on the permutation group $\F S_n$ (i.e. the Coxeter group of
type \textsc{\MakeLowercase{A}}) was introduced, which we shall call the \textit{colour-length function}
$\|\cdot\|$. It is defined as follows: for $w\in \F S_n$ in the minimal
(reduced) representations $w = s_1\dots s_k$, where each $s_{j}$ belong to the set $S$
of transposions of the form  $(i,i+1)$, we put $\|w\| = \#\{s_{1}, s_{2},\dots, s_{k}\}$.

For our study one of  the most important results
of this paper is that the function
called Riesz-Coxeter product $R_{\B q}$ defined on all Coxeter groups
$(W,S)$ as
$$
R_{\B q}(s) = q_{s} \text{, for $s\in S$, and\ }
R_{\B q}(xy) = R_{\B q}(x) R_{\B q} (y) \text{, if\ } \| x y \| = \| x \| + \|y\|
$$
is positive definite for $0\leq q_s\leq 1$.

This implies, in particular, that in an arbitrary Coxeter group the set
of its Coxeter generators is a weak Sidon
set and also it is completely bounded $\Lambda^\cb_p$-set, see
Theorems \ref{T:5.1} and
\ref{T:6.1}.  Equivalently, the span of the linear operators $\{\lambda(s)|s\in
S\}$ in the noncommutative $L^{p}$-space $L^p(W)$ is completely boundedly isomorphic
to row and column operator Hilbert space (see Theorem \ref{T:6.1}).


Another interesting connection between the two length functions $|\cdot|$ and
$\|\cdot\|$ appeared in \cite{MR1388006} in the formula for the moments of
free additive convolution power of the Bernoulli law $\mu_{-1} =
(\delta_{-1} + \delta_1)/2$ (cf.~Corollary~6 in cited paper):
$$
m_{2n}\left(\mu_{-1}^{\boxplus q}\right)
=q^n\smashoperator{\sum_{\pi\in\C P_2(2n)}}(-1)^{|\pi|}q^{-\|\pi\|},
$$
for $q\in\B N$.  (See also Section 9 of the present paper.)

Also, in \cite{MR2764902} the colour-length function on the permutation group
$\F S_n$ was studied. Some of its extensions to pairpartitions appeared
in the presentation of the proof that classical normal law $N(0,1)$ is free
infinitely divisible under free additive convolution $\boxplus$.

Since we have recent extensions of the free probability (which is related to
type \textsc{\MakeLowercase{A}} Coxeter groups) to the free probability of type \textsc{\MakeLowercase{B}} Coxeter groups
(see \cite{MR3373433}), it seems to be
interesting to determine the role of the colour-length functions for the Coxeter
groups of type \textsc{\MakeLowercase{B}} and \textsc{\MakeLowercase{D}}.

The plan of the paper is as follows.

In Section 1 we recall definitions of Coxeter groups and of the length and the
colour-length functions.

In Section 2 we recall the definition of positive definite funcios and discuss
various classes of those, namely radial, colour-radial, and colour-dependant.

In Section 3 we discuss Abelian Coxeter groups.

In Section 4 we show the following formula characterizing the radial normalised
positive definite functions on these Coxeter groups which contain the infinite
Rademacher group $\bigoplus_{i=1}^\infty \B Z/2$
as a parabolic subgroup (these include the infinite
permutation group $\F S_\infty$): every radial positive definite function
$\varphi$ is of the form
$$
\varphi (w) =\int_{-1}^1 q^{|w|}\,\mu(dq)
$$
for a probability measure $\mu$.

That characterisation is a variation on the classical de Finetti theorem.
A noncommutative version was shown by K\"ostler ans Speicher \cite{MR2530168}
(see also \cite{MR2092722}).

We also show in Theorem \ref{T:frac}, that the function $\exp(-t|w|^p)$
is positive definite for all $t\geq 0$ if and only if $p\in[0,1]$.

In Section 5 we give a short proof of the equivalence of the two known results
concerning positive definite functions on finite Coxeter groups.

In Section 6 we present the main properties of the colour-dependent positive definite
functions on Coxeter groups, in particular we show in Proposition 4.4. that on
$\F S_\infty$ and some other Coxeter groups, the function $w\mapsto r^{\|w\|}$
is positive definite if and only if $r\in[0,1]$.

The Section 7 gives characterization of all colour-length functions on finite and
infinite dihedral groups $\B D_m$, for $m=1,2,\dots,\infty$.

In Section 8 we prove that the set $S$ of Coxeter generators is a weak Sidon
set in arbitrary Coxeter groups $(W,S)$ with constant 2 and that it is also a
completely bounded $\Lambda (p)$ set with contants as $C \sqrt p$, for $p>2$.

In Section 9 we prove for arbitrary finitely generated Coxeter group an identity
involving both lengths $|\cdot|$ and $\|\cdot\|$ (see Proposition \ref{P:tokyo}).
We apply it to give a proof of Corollary 7 from \cite{MR1388006},
(see Equation (\ref{eq:p2-nc2})) where the proof, involving probabilistic considerations,
was not presented in \cite{MR1388006}.

\section{Coxeter groups}

In this part we recall the basic facts regarding Coxeter groups and introduce notation
which will be used throughout the rest of the paper. For more details we refer to \cite{MR0240238,MR1066460}.

A group $W$ is called a \emph{Coxeter group} if it admits the following presentation:
\[
W=\left\langle S\left|\left\{(s_1 s_2)^{m(s_1,s_2)}=1:s_1,s_2\in S,m(s_1,s_2)\ne\infty\right\}\right.\right\rangle,
\]
where $S\subset W$ is a set and $m$ is a function $m:S\times S\to\{1,2,3,\ldots,\infty\}$
such that $m(s_1,s_2)=m(s_2,s_1)$ for all $s_1,s_2\in S$ and $m(s_1,s_2)=1$ if and only if $s_1=s_2$.
The pair $(W,S)$ is called a \emph{Coxeter system}.
In particular, every generator $s\in S$ has order two
and every element $w\in W$ can be represented as
\begin{equation}\label{areducedrepr}
w=s_1 s_2\ldots s_{m}
\end{equation}
for some $s_1,s_2,\ldots,s_m\in S$.
If the sequence $(s_1,\ldots,s_m)\in S^{m}$ is chosen in such a way that $m$ is minimal
then we write $|w|=m$ and call it the \emph{length} of $w$.
In such a case the right hand side of (\ref{areducedrepr}) is called a \emph{reduced representation}
or \emph{reduced word} of $w$.
This is not unique in general, but the set of involved generators is unique \cite[Ch.~\textsc{\MakeLowercase{iv}}, \S1, Prop.~7]{MR0240238}, i.e.
if $w=s_1 s_2\ldots s_m=t_1 t_2\ldots t_m$ are two reduced representations
of $w\in W$ then $\{s_1,s_2,\ldots,s_m\}=\{t_1,t_2,\ldots,t_m\}$.
This set $\{s_1,s_2,\ldots,s_m\}\subseteq S$ will be denoted $S_w$ and called the \emph{colour} of $w$.

Given a subset $T\subset S$ by $W_T$ we denote the subgroup generated by $T$ and call it the \emph{parabolic
subgroup} associated with $T$.
To see that $ S_w$ is independent of the reduced representation of $w$ notice that
\begin{equation}
\label{def:colour}
s\in S_w \iff w\not\in W_{S\smallsetminus\{s\}}.
\end{equation}
We define the \emph{colour-length} of $w$ putting $\|w\|=\# S_w$ (the cardinality of $S_w$).
Both lengths satisfy the triangle inequality and we have $\|w\|\le |w|$.

In the case of the permutation group the colour-length has the following pictorial
interpretation. If $\sigma$ is a permutation
in $\F S_{n+1}$ then $\|\sigma\|$ equals $n$ minus the number
of connected components of the diagram representing~$\sigma$.
Notice, that $|\sigma|$ equals to the number of crossings in the diagram
(the number of pairs of chords that cross).

\begin{tikzpicture}
\newcommand{\perm}[1]{\draw \foreach \x/\y in {#1} {(\x*.75,1)--(\y*.75,0)};}
\newcommand{\xtext}[1]{\draw (1.5,.5) node[fill=white] {#1};}
\matrix[column sep=0.8cm,row sep=0.5cm]{
\xtext {$\sigma$};&\xtext {$e$};&\xtext {$(12)$};&\xtext {$(12)(23)$};&\xtext {$(12)(23)(12)$};\\
&\perm{1/1,2/2,3/3}&\perm{1/2,2/1,3/3}&\perm{1/2,2/3,3/1}&\perm{1/3,2/2,3/1}\\
\xtext {$|\sigma|$};&\xtext {0};&\xtext {1};&\xtext {2};&\xtext {3};\\
\xtext {$\|\sigma\|$};&\xtext {0};&\xtext {1};&\xtext {2};&\xtext {2};\\};
\end{tikzpicture}

%
%

It would be convenient to define 
\begin{equation}
\label{def:block-length}
\|w\|_s\colon=
\begin{cases}
0&\text{if $s\not\in S_w$,}\\
1&\text{if $s\in S_w$,}
\end{cases}
\end{equation}
then, clearly, $\|w\|=\sum_{s\in S}\|w\|_s$.

\def\Rad{\mathop{\rm Rad}\nolimits}

\section{Positive defined functions}

A complex function $\varphi$ on a group $\Gamma$ is called
\emph{positive definite}
if we have
\[
\sum_{x,y\in \Gamma}\varphi(y^{-1}x)\alpha(x)\overline{\alpha(y)}\ge0
\]
for every finitely supported function $\alpha\colon\Gamma\to\B C$.

A positive definite $\varphi$ function is Hermitian and satisfies $|\varphi(x)|\le\varphi(e)$
for all $x\in \Gamma$. Usually it is assumed, that $\varphi$ is \emph{normalised}, i.e. that $\varphi(e)=1$.

In this and the following sections we discuss the \textit{radial functions} on Coxeter groups.
These are functions which depend on $|w|$ rather then on $w$.

We call a function $\varphi$ on $(W,S)$ \emph{colour-dependent} if
$\varphi(w)$ depends only on $S_w$.  We call it \emph{colour-radial}
if it depends only on $\|w\|$.

An Abelian Coxeter group generated by $S$ is isomorphic to the direct product $\oplus_{s\in S}\B Z/2$.
On these groups the lengths $|\cdot|$ and $\|\cdot\|$ coincide and all functions are colour dependent.

The main example of a positive definite function will be the \emph{Riesz--Coxeter function}.  Given a sequence $\B q=(q_s)_{s\in S}$
we define $R_{\B q}(w)=\prod_{s\in S}q_s^{\|w\|_s}=\prod_{s\in  S_w}q_s$.
We will abuse notation and denote by $R_{\B q}$ also the associated operator
$\sum_{w\in W}R_{\B q}(w)w$.  That is
$$
R_{\B q}=1+\sum_{s\in S}q_ss+\sum_{w: S_w=\{s_1,s_2\}}q_{s_1}q_{s_2}w+\sum_{w: S_w=\{s_1,s_2,s_3\}}q_{s_1}q_{s_2}q_{s_3}w+\dots
$$

In the case all $q_s=q$ we get $R_{\B q}=\sum q^{\|w\|}w$.

This generalises the classical case of Rademacher--Walsh functions in the
Rademacher group $\Rad_n$.  If we denote the generator of the $i$-th factor $\B Z/2$
of the latter by the symbol $r_i$ then, by definition, $r_1^2=1$ and $r_ir_j=r_jr_i$ and
$$
R_{\B q}=\prod\limits_{i=1}^n(1+q_ir_i).
$$

\section{Rademacher groups}


In this section we are going to study positive definite radial functions on the Abelian Coxeter groups,
$(W,S)=\Rad_{S}$.
Since positive definiteness is tested on functions with finite support, we can assume that $S$ is countable.
If $\#S=n$ we will write $\mathrm{Rad}_{n}$ instead of $\mathrm{Rad}_{S}$.
Given $n\in\B N\cup\{\infty\}$, we denote by $P_n$ the class of all functions $f\colon\{0,1,\dots,n\}\to\B R$ for $n$ finite and $f\colon\B N\to\B R$ if
$n=\infty$ such that
$\varphi(w)=f(|w|)$ is a normalised positive definite on $\Rad_n$.

The following observation is straightforward.

\begin{proposition}
Assume that $1\leq m< n\leq\infty$ and $f\in P_n$.  Then the restriction of $f$ to $\{0,\dots,m\}$
belongs to $P_m$.  A fuction $f$ belogs to $P_\infty$ if and only if all its restrictions to $\{0,\dots,m\}$
for any $m\in\B N$ belong to $P_m$.
\end{proposition}

\begin{theorem}
Assume $n$ is finite.
The set $P_n$ form a simplex whose
vertices (extreme points) are $f^n_l(k)={n\choose l}^{-1}\sum_{i=0}^l(-1)^i{k\choose i}{n-k\choose l-i}$,
where $0\leq l\leq n$.  Equivalently, every normalised radial positive definite function on the group $\Rad_n$
is of the form
$$
\varphi(x)=\sum_{l=0}^n\lambda_lf^n_l(|x|),
$$
where the sequence of nonnegative numbers $(\lambda_l)_{l=0}^n$ is unique and satisfies
$\sum_{l=0}^n\lambda_l=1$.
\end{theorem}

\begin{proof}
We can indentify the dual $\widehat{\Rad_n}$ group of $\Rad_n$ with $\Rad_n$ via the paring
$(x,y)=(-1)^{\sum_{i=1}^n x_iy_i}$.  By Bochner's theorem every normalised
positive definite function $\varphi$ on $\Rad_n$ is of the form
\[
\varphi(x)=\int_{\widehat\Rad_{\infty}}(x,y)\,\mu(dy),
\]
for some probability measure $\mu$.  Clearly, such a function
is radial if and only if $\mu$ is invariant under the action of $\F S_n$.

Among such measures extreme ones are measures $\mu_l$ for $0\leq l\leq n$,
where $\mu_l$ is equally distributed among elements of length $l$.
Moreover,
\[
\varphi(x)=\int_{\widehat\Rad_{\infty}}(x,y)\,\mu_l(dy)
=f^n_l(|x|)
\]
as claimed.
\end{proof}

The following theorem is a version of the classical de Finetti Theorem
(see \cite[p. 223]{MR0270403}) for the infinite Rademacher group.

\begin{theorem}
Assume that $\varphi$ is a radial function on the Rademacher group
$\Rad_\infty$.
Then $\varphi$ is a normalised positive definite if and only if there exists
a probability measure $\mu$ on $[-1,1]$ such that
$$
\varphi(x)=\int_{-1}^1 q^{|x|}\mu(dq).
$$
This measure $\mu$ is unique.
\end{theorem}

\begin{proof}
Since the function $q^{|x|}$ is normalised positive definite for $q\in[-1,1]$, the ``if'' implication is obvious.

Assume that $\varphi$ is normlised positive definite.
The group $\Rad_{\infty}$ is discrete and Abelian and its dual is the compact group
$\widehat{\Rad_\infty}=\prod_{i=1}^{\infty}\B Z/2$.
By Bochner's theorem, there exists a probability measure $\eta$ on $\widehat\Rad_{\infty}$
such that
\[
\varphi(x)=\int_{\widehat\Rad_{\infty}}(x,y)\,d\eta(y),
\]
where for $x=(x_1,x_2,\ldots)\in \Rad_\infty$, $y=(y_1,y_2,\ldots)\in \widehat\Rad_{\infty}$
we put $(x,y)=(-1)^{\sum_{i=1}^\infty x_i y_i}$.
The radiality of $\varphi$ is equivalent to the fact that for every permutation $\sigma\in\mathfrak{S}_\infty$
we have $\varphi(x)=\varphi(\sigma(x))$, where $\sigma(x)=(x_{\sigma(1)},x_{\sigma(2)},\ldots)$.
This, in turn, implies that $\eta$ is $\sigma$-invariant for every $\sigma\in\mathfrak{S}_\infty$,
i.e. we have $\eta(A)=\eta(\sigma(A))$ for every Borel subset $A\subset\widehat\Rad_{\infty}$.

For a sequence $\B \epsilon=(\epsilon_i)_{i=1}^n\in\{0,1\}^n$ we define
$C_{n}(\B \epsilon)\subseteq\widehat\Rad_{\infty}$ by
\[
C_{n}(\B \epsilon)=\{y\in\widehat\Rad_{\infty}|y_i=\epsilon_i:1\leq i\leq n\},
\]
in particular $C_0(\varnothing)=\widehat\Rad_{\infty}$.
Then we have $\eta(C_{n}(\epsilon))=\eta(C_{n}(\B\epsilon'))$
if $\epsilon'_i=\epsilon_{\sigma(i)}$ for some $\sigma\in\mathfrak{S}_n$
and every $1\leq i\leq n$.
For $\varepsilon \in \B R$ we put
\[
\varepsilon^n=\underbrace{\varepsilon,\varepsilon,\ldots,\varepsilon}_{n} 
\]
and $a_n=\eta(C_{n}(1^{n}))$. Moreover, for $n,k\ge0$ we define the difference operators $\Delta^{k}a_n$ by induction:
$\Delta^{0}a_n=a_n$ and $\Delta^{k+1}a_n=\Delta^{k}a_{n+1}-\Delta^{k}a_{n}$.
We claim that
\begin{equation}\label{deltakn}
(-1)^{k}\Delta^{k}a_{n}=\eta\left(C_{n+k}\left(1^n0^k\right)\right).
\end{equation}
Denoting the right hand side of (\ref{deltakn}) by $c(n,k)$ we note that $c(n,0)=a_n$ and
\[
C_{n+k+1}\left(1^n0^k0\right)\cup
C_{n+k+1}\left(1^n0^k1\right)
=C_{n+k}\left(1^n0^k\right),
\]
is a disjoint union.  This implies
\[
c(n,k+1)=c(n,k)-c(n+1,k).
\]
This formula, by induction on $k$, leads to (\ref{deltakn}).

From (\ref{deltakn}) we see that the sequence $(a_n)$ is \textit{completely monotone},
i.e. that $(-1)^{k}\Delta^{k}a_n\ge0$ for all $n,k\ge0$.
By the celebrated theorem of Hausdorff (see \cite[S\"atze \textsc{\MakeLowercase{ii}} und \textsc{\MakeLowercase{iii}}]{MR1544467}),
there exists a unique probability measure
$\rho$ on $[0,1]$ such that
\begin{equation}
\label{eq-Haus}
(-1)^{k}\Delta^{k}a_n=\int_{0}^{1} u^n(1-u)^k\,d\rho(u).
\end{equation}
(Note that Equation (\ref{eq-Haus}) for arbitrary $k\geq 0$ follows from the case $k=0$.)
 
For $x=\left(1^n0^\infty\right)\in \Rad_\infty$ so that $|x|=n$, we have
\begin{align*}
\varphi(x)&=\int_{\widehat\Rad_{\infty}}(x,y)\,d\eta(y)
=\int_{\widehat\Rad_{\infty}}(-1)^{\sum_{i=1}^ny_i}\,d\eta(y)\\
&=\sum_{\epsilon\in\{0,1\}^{n}}(-1)^{\sum_{i=1}^n\epsilon_i}\eta(C_{n}(\epsilon))
=\sum_{k=0}^{n}\binom{n}{k}(-1)^{k}\eta\left(C_n\left(1^k0^{n-k}\right)\right)\\
&=\sum_{k=0}^{n}\binom{n}{k}(-1)^{k}\int_{0}^{1} u^k(1-u)^{n-k}\,d\rho(u)
=\int_{0}^{1}(1-2u)^{n}\,d\rho(u)\\
&=\int_{-1}^{1}q^n \,d\mu(q),
\end{align*}
where $\mu$ is defined by $\mu(A)=\rho\left(\frac{1}{2}+\frac{1}{2}A\right)$ for a Borel set $A\subseteq[-1,1]$.
\end{proof}

\section{Remarks on radial positive definite functions on some infinitely generated Coxeter groups}

In this Section we extend the last theorem of the previous section to a certain class of Coxeter groups.

\begin{theorem}
Assume that $(W,S)$ is a Coxeter system and that there is an infinite subset $S_0\subseteq S$
such that $st=ts$ for $s,t\in S_0$.
Assume that $\varphi$ is a radial function on $W$ with $\varphi(e)=1$.
Then $\varphi$ is positive definite if and only if there exists
a probability measure $\mu$ on $[-1,1]$ such that
$$
\varphi(\sigma)=\int_{-1}^1 q^{|\sigma|}\mu(dq).
$$
This measure $\mu$ is unique.
\end{theorem}

\begin{proof}
It is sufficient to note that the group generated by $S_0$ is a parabolic subgroup isomorphic with
$\Rad_\infty$.
\end{proof}

\textbf{Example.}
For $W=\mathfrak{S}_\infty$ we have $S=\{(n,n+1):n\in\B N\}$.
Then we can take $S_0=\{(2n-1,2n):n\in\B N\}$.  Similar
$S_0$ can be found in infinitely generated groups of type \textsc{\MakeLowercase{B}} and \textsc{\MakeLowercase{D}}.

\begin{problem}
When $-1\leq q\leq1$, $q\neq 0$ is the positive definite function
$q^{|x|}$ on $\F S_\infty$ an extreme point in the set of
normalised positive definite functions?
\end{problem}

\begin{theorem}\label{T:frac}
The function $\varphi_p(\sigma)=e^{-t|\sigma|^p}$ is positive definite on $\F S_\infty$
if and only if $0\leq p\leq1$.
\end{theorem}

\begin{proof}
{\em A contrario.}  Assume that for some $p>1$ and $t_0>0$ the function $\psi_p(\sigma)=e^{-t_0|\sigma|^p}$
is positive definite on $\F S_\infty$.

For $q_0=e^{-t_0}$, choosing $\sigma$ such that $|\sigma|=2n$ we have
$q_0^{(2n)^p}=\int_{-1}^1q^{2n}\,d\mu_0(q)$ for some probability measure $\mu_0$
on $[-1,1]$.  Since $\left(\int_{-1}^1q^{2n}\,d\mu_0(q)\right)^{1/n}$ tends to $\max\{q^2|q\in\mathop{\hbox{\rm supp}}\mu_0\}$
while $\left(q_0^{(2n)^p}\right)^{1/n}\to 0$, we conclude that $\mu_0$ is the Dirac measure at $0$,
which is a contradiction.

The ``if'' part is standard.
We need to show that $f(x)=e^{-tx^p}$ is the Laplace transform of some probability measure supported on $[0,\infty)$,
so $f$ is a convex combination of functions of the form $e^{-sx}$.

By characterisation of Laplace transforms (see \cite[Satz \textsc{\MakeLowercase{iii}}]{MR1544467})
this is equivalent to \emph{complete monotonicity}, that is $(-1)^nf^{(n)}>0$ for all $n=0,1,\dots$.
And indeed, by induction, $(-1)^nf^{(n)}$ is a positive linear
combination of positive functions of the form $x^{pj-n}f(x)$ for $1\leq j\leq n$.
\end{proof}

The measures with Laplace transforms $e^{-tx^p}$ for $t\geq0$ and $0\leq p\leq 1$
are studied in detail in \cite[Ch.~\textsc{\MakeLowercase{ix}.11}]{MR617913}
(see Propositions~1 and 2 there).

Let us note that for such groups like $\B Z^k$ or $\B R^k$ with the Euclidean distance $d$
the functions $\exp(-td^p)$ are positive definite for all $t\geq 0$ and $0\leq p\leq 2$
(the case $p=2$ corresponds to the Gau\ss ian Law).

\section{The longest element}

If a Coxeter group $W$ is finite, then it contains the unique element $\omega_\circ$
which has the maximal length with respect to $|\cdot|$.

From the definition it is clear, that a function $\varphi$ on a group $\Gamma$ with values in the field
of complex numbers $\B C$ is {\em positive definite}
if and only $\sum_{g\in\Gamma}\varphi(g)g$ is a nonnegative (bounded if the group is finite)
operator on $\ell^2\Gamma$.  (We will identify $g\in \Gamma$ with $\lambda(g)\in B(\ell^2\Gamma)$,
where $\lambda$ is the left regular representation, for short.)

Let $W$ be a finite Coxeter group.  The following two statements are well known.
\begin{enumerate}[label={(\textsc{\alph*})}]
\item The function $q^{|w|}$ is positive definite for any $0\leq q\leq 1$.
\item The function $\Delta(w)=|\omega_\circ|/2-|w|$ is positive definite.
\end{enumerate}
The first one was proven in \cite{MR950825} (even for infinite Coxeter groups and also for
$-1\leq q\leq 1$) while the second --- in \cite[Proposition~6]{MR2009841}.
Here we give a short direct prove of the following.

\begin{proposition}\label{P: }
The above statements (\textsc{a}) and (\textsc{b}) are equivalent.
\end{proposition}

\begin{proof}
Let $q=e^{-t}$ (with $t\geq0$, as we assume $q\leq 1$).
The case (\textsc{a}) is equivalent to $\Phi_t=\sum_{w\in W}e^{t\Delta(w)}w=e^{t|\omega_\circ|/2}\sum_{w\in W}q^{|w|}w$
being nonnegative.

Assume (\textsc{a}).
Recall first, that $|\omega_\circ w|=|\omega_\circ|-|w|=|w\omega_\circ|$.  Therefore
$|\omega_\circ|/2-|\omega_\circ w|=-(|\omega_\circ|/2-|w|)$, ie.
$\Delta(\omega_\circ w)=-\Delta(w)$ and similarly, $\Delta(w\omega_\circ)=-\Delta(w)$.

The equality $\Delta(\omega_\circ w)=\Delta(w\omega_\circ)$ implies that $\omega_\circ$
(and thus $Q=(1-\omega_\circ)/2$) commutes with $\Delta$ (and thus $\Phi_t$).
Since $Q=Q^2$ is nonnegative we conclude that
$$
t^{-1}\Phi_tQ=\sum_{w\in W}\frac{e^{t(|\omega_\circ|/2-|w|)}-e^{t(|\omega_\circ|/2-|w\omega_\circ|)}}{2t}w
=\sum_{w\in W}\frac{\sinh(t\Delta(w))}{t}w.
$$
is nonnegative.  Therefore, taking the limit as $t\to0$, we obtain that
$\sum_{w\in W}\Delta(w)w$ is nonnegative.
Thus (\textsc{b}).

Assuming (\textsc{b}) and using the Schur lemma, which says that the (pointwise)
product of positive definite functions is positive definite, we get that
$$
\Phi_t=\sum_{w\in W} e^{t\Delta(w)}w
=\sum_{n\geq 0}\frac{t^n}{n!}\left(\sum_{w\in W}\Delta(w)^nw\right)
$$
is nonnegative.
Thus (\textsc{a}).
\end{proof}

\section{Colour-dependent positive definite functions on Coxeter groups}

The question which colour-dependant or colour-radial functions
are positive functions on Coxeter groups is wide open.
In this section we provide some sufficient conditions.
In the next section we will
examine the dihedral groups in full details.

\begin{lemma}\label{alemmasubgroup}
Let $H$ be a subgroup of a group $\Gamma$ of index $d$.
Then the function $\varphi_r$ defined by $\varphi_{r}(x)=1$ if $x\in H$
and $\varphi_{r}(x)=r$ otherwise is positive definite on $\Gamma$
if and only if $r\in[-1/(d-1),1]$,
with natural convention that if $d=\infty$ then $-1/(d-1)=0$.
\end{lemma}

Note, that if $H=\{1\}$ then $d=|\Gamma|$.

\begin{proof}
First assume that $d$ is finite and let us enumerate the left cosets:
\[
\{gH:g\in \Gamma\}=\{H_1,H_2,\dots,H_d\}.
\]
Note, that for $x\in H_i, y\in H_j$ we have $y^{-1}x\in H$
if and only if $i=j$. Therefore, for $r_0=-1/(d-1)$
and for a finitely supported
complex function $f$ on $\Gamma$ we have
\[
\sum_{x,y\in \Gamma} \varphi_{r_0}(y^{-1}x)f(x)\overline{f(y)}
=\frac{1}{d-1}\sum_{1\le i<j\le d}
\left|\sum_{x\in H_i}f(x)-\sum_{y\in H_j}f(y)\right|^2,
\]
which proves that $\varphi_{r_0}$ is positive definite.
For $r\in[-1/(d-1),1]$ the function $\varphi_r$
is positive definite as a convex combination of $\varphi_{r_0}$
and the constant function $\varphi_1$.

On the other hand, if we choose $x_i\in H_i$ for each $i\le d$
and define
$f$ as the characteristic function of the set $\{x_1,\ldots,x_{d}\}$
then
\begin{equation}\label{apositivedefiniterestriction}
\sum_{x,y\in \Gamma}\varphi_r(y^{-1}x)f(x)\overline{f(y)}
=d+(d^2-d)r,
\end{equation}
which proves that $r\ge -1/(d-1)$ is a necessary condition
for positive definiteness of $\varphi_r$.

If $d=\infty$ then $r_0=0$ and the function $\varphi_0$
is positive definite as the characteristic function of the subgroup $H$.
For ``only if" part we chose an arbitrarily long sequence $x_1,\ldots,x_{d'}$ of elements
from different left cosets and use (\ref{apositivedefiniterestriction}) with $d'$ instead of $d$.
\end{proof}

\begin{theorem}\label{bpropositionproduct}
Assume that for every $s\in S$ we are given a number $q_s$,
\[
\frac{-1}{d_{s}-1}\le q_s\le1,
\]
where $d_{s}$ denotes the index if the parabolic subgroup generated by $S\smallsetminus\{s\}$ in $W$:
$d_{s}=[W:W_{S\smallsetminus\{s\}}]$.
Then the Riesz--Coxeter $R_{\B q}$ is positive definite on $W$.
\end{theorem}

\begin{proof}
From Lemma~\ref{alemmasubgroup} the function $w\mapsto q_{s}^{\|w\|_s}$
is positive definite for $s\in S$ and $-1/(d_{s}-1)\le q_s\le1$.
Since the pointwise product of positive definite functions is positive definite,
the statement holds.
\end{proof}

\textbf{Example.} Take $W=\mathfrak{S}_n$, the permutation group on the set $\{1,2,\ldots,n\}$.
It is generated by the transpositions $S=\{s_i=(i,i+1), 1\le i\le n-1\}$.
For $1\le i\le n-1$ the parabolic subgroup generated by $S\smallsetminus\{s_i\}$
is isomorphic with $\mathfrak{S}_{i-1}\times \mathfrak{S}_{n-i-1}$, so its index
is $i{n\choose i}$.

It would be interesting to determine for which $r$ the function
$w\mapsto r^{\|w\|}$ is positive definite. By Proposition~\ref{bpropositionproduct}
this holds for $r\in[-1/(d-1),1]$, where $d$ is the maximal index of the parabolic subgroups
of the form $W_{S\smallsetminus\{s\}}$.
We note a necessary condition.

\begin{proposition}
Assume that we have distinct generators $s_0,s_1,\dots,s_n\in S$
such that $s_0s_k\ne s_ks_0$ (i.e. $m(s_0,s_k)>2$) for $1\le k\le n$.
If the function $w\mapsto r^{\|w\|}$ is positive definite on $W$,
then $\sfrac{-1}{(n-1)}\le r^3\le1$.

If there is an element $s_0\in S$ for which there are infinitely many $s\in S$
such that $s_0s\ne ss_0$ then $ r^{\|w\|}$ is positive definite on $W$ if and only if\/ $0\le r\le 1$.
\end{proposition}

\begin{proof}
Consider elements $w_k=s_0s_ks_0$.
Note, that for $k\ne l$ we have $\|w_l^{-1}w_k\|=3$.
If $\varphi_r$ is positive definite on $W$
then we have
\[
0\le\sum_{k,l=1}^{n}\varphi_r(x_l^{-1}x_k)=
n+(n^2-n)r^3,
\]
which implies $r^3\ge\sfrac{-1}{(n-1)}$.
\end{proof}

\begin{corollary}
The function $w\mapsto q^{\|w\|}$ on $\F S_\infty$ is positive definite if and only if\/ $0\leq q\leq 1$.
\end{corollary}

\begin{problem}
Thus, it is valid to ask the following.
Is it true that {\em every} normalised positive definite colour-lenght-radial
function $\phi\colon\F S_\infty\to\B R$ is of the form $\phi(\sigma)=\int_0^1 q^{\|\sigma\|}\,d\mu(q)$
for some probability measure $\mu$ on $[0,1]$?
\end{problem}

\section{Dihedral groups}

In this part we are going to examine the class of colour-dependent
positive definite functions on the case the simplest nontrivial Coxeter groups.
Assume that $W=\B D_{2n}=\langle s,t|(st)^n\rangle$ (i.e.\ the group of symmetries of a regular
$n$-gon), and define a colour-dependent function  on $W$:
\begin{equation}\label{functiondichedralpqr}
\phi(w)=
\begin{cases}
1&\text{if $w=e$,}\\
p&\text{if $w=s$,}\\
q&\text{if $w=t$,}\\
r&\text{otherwise.}
\end{cases}
\end{equation}
If $p=q$ then $\phi$ is colour radial.
We are going to determine for which parameters $p,q,r$ the function $\phi$ is positive definite on $W$.
It is easy to observe necessary conditions: $p,q,r\in[-1,1]$.
Moreover, since $\left\langle st\right\rangle$
is a cyclic subgroup of order $n$, Lemma~\ref{alemmasubgroup},
implies a necessary condition: $-1/(n-1)\le r\le1$.

\subsection*{Finite dihedral groups}

Assume that $W$ is a finite dihedral group, $W=\B D_{2n}$,
so that $(st)^n=1$.
We will use the following version of Bochner's theorem:
A function $f$ on a compact group $G$ is positive definite
if and only if its \emph{Fourier transform}:
\[\widehat{f}(\pi)=\int_{G} f(x)\pi(x^{-1})dx
\]
is a positive operator for every $\pi\in\widehat{G}$,
where $\widehat{G}$ denotes the dual object of $G$, i.e. the family of
all equivalency classes of unitary irreducible representations of $G$, see~\cite{MR1363490}.
Then we have
\[
f(x)=\sum_{\pi\in\widehat{G}}d_{\pi}\tr\left[\widehat{f}(\pi)\pi(x)\right].
\]

Therefore, for every irreducible representation $\pi$ of $\B D_{2n}$
we are going to find
\[
\widehat{\phi}(\pi)=\frac{1}{2n}\sum_{g\in G}\phi(g)\pi(g^{-1}).
\]

We will identify $s$ with $(0,-1)$ and $t$ with $(1,-1)$.
If $n$ is odd then $\B D_{2n}$ possesses two characters:
$\chi_{+,+}$ such that $\chi_{+,+}(w)=1$ for every $w\in\B D_{2n}$
and $\chi_{-,-}$ such that $\chi_{-,-}(s)=\chi_{-,-}(t)=-1$.
If $n$ is even then we have two additional characters
$\chi_{+,-}$ and $\chi_{-,+}$
such that $\chi_{+,-}(s)=\chi_{-,+}(t)=1$ and $\chi_{+,-}(t)=\chi_{-,+}(s)=-1$.
It is easy to check that
\[
2n\widehat{\phi}(\chi_{+,+})=1+p+q+(2n-3)r,
\]
\[
2n\widehat{\phi}(\chi_{-,-})
=1-p-q+r,
\]
which gives
\[
-1-(2n-3)r\le p+q\le 1+r
\]
and, for $n$ even,
\[
2n\widehat{\phi}(\chi_{+,-})=1+p-q-r,
\]
\[
2n\widehat{\phi}(\chi_{-,+})=1-p+q-r,
\]
which implies
\[
|p-q|\le 1-r.
\]
We have also the family of two dimensional representations $U_a$:
\begin{align*}
U_a(k,1)&=\begin{pmatrix}
e^{2\pi ika/n}&0\\
0&e^{-2\pi ika/n}
\end{pmatrix},\\
U_a(k,-1)&=\begin{pmatrix}
0&e^{2\pi ika/n}\\e^{-2\pi ika/n}&0
\end{pmatrix},
\end{align*}
where $a=1,2,\dots,\left\lfloor\frac{n-1}{2}\right\rfloor$.
Then for the function given by (\ref{functiondichedralpqr}) we have
\begin{align*}
2n\widehat{\phi}(U_a)&=(1-r)\mathrm{Id}+(p-r)U_a(0,-1)+(q-r)U_a(1,-1)\\
&=\begin{pmatrix}
1-r&p-r+(q-r)e^{2\pi ia/n}\\p-r+(q-r)e^{-2\pi ia/n}&1-r
\end{pmatrix}.
\end{align*}
This matrix is positive definite if and only if $r\le 1$
and
\[
\left|p-r+(q-r)e^{2\pi ia/n}\right|\le 1-r.
\]
Therefore we have
\begin{proposition}\label{propositiondihedralpqr}
The function $\phi$ given by (\ref{functiondichedralpqr}) is positive definite on $\B D_{2n}$
if and only if $$1+p+q+(2n-3)r\ge0,\qquad 1-p-q+r\ge0$$
(plus $$1+p-q-r\ge0,\qquad 1-p+q-r\ge0$$ whenever $n$ is even)
and
\[
\left|p-r+(q-r)e^{2\pi ia/n}\right|\le 1-r.
\]
for $a=1,2,\dots,\left\lfloor\frac{n-1}{2}\right\rfloor$.
\end{proposition}

Let us confine ourselves to colour-radial functions.

\begin{corollary}
Assuming that $p=q$, the function $\phi$ defined by (\ref{functiondichedralpqr})
is positive definite on $W=\B D_{2n}$ if and only if
\[
\max\left\{\frac{-2p-1}{2n-3},2p-1\right\}\le r\le \frac{1+2p\cos(\sfrac{\pi}n)}{1+2\cos(\sfrac{\pi}n)},
\]
i.e. if and only if the point $(p,r)$ belongs to the triangle whose vertices are
\[
\left(\frac{1-n-\cos(\sfrac{\pi}n)}{1+(2n-1)\cos(\sfrac{\pi}n)},\frac{1-\cos(\sfrac{\pi}n)}{1+(2n-1)\cos(\sfrac{\pi}n)}\right),\quad
\left(\frac{n-2}{2n-2},\frac{-1}{n-1}\right),\quad(1,1).
\]
\end{corollary}

\begin{proof}
For $p=q$ the conditions from Proposition~\ref{propositiondihedralpqr} reduce to 
\[
2p-1\le r,\quad
-1-2p\le (2n-3)r,\quad\mbox{ and }\quad
2\cos(\sfrac{\pi}n)|p-r|\le 1-r.
\]
It is sufficient to note that $2p-1\le r$ implies
$2\cos(\sfrac{\pi}n)(p-r)\le1-r$ for $p\le1$.
\end{proof}

\textbf{Example.} For $\B D_4$ we have the positive definiteness of $\phi$ is equivalent to
\[
-1+|p+q|\le r\le 1-|p-q|,
\]
which means that the set of all possible $(p,q,r)$ forms a tetrahedron
with vertices $(-1,1,-1)$, $(1,-1,-1)$, $(-1,-1,1)$, $(1,1,1)$.
For $p=q$ the condition reduces to $2|p|-1\le r\le1$.

In the case of $\B D_{6}$ Proposition~\ref{propositiondihedralpqr} leads to the following conditions:
\[
1-p-q+r\ge0,\qquad 1+p+q+3r\ge0,
\]
\[
1-r\ge\sqrt{p^2+q^2+r^2-pq-pr-qr},
\]
which can be expressed as
\[
\max\left\{\frac{-1-p-q}{3},p+q-1 \right\}\le r\le\frac{1-p^2-q^2+pq}{2-p-q}.
\]

\subsection*{The infinite dihedral group}

%
%
%
%
%
%
%

Here we are going to study $W=\B D_{\infty}$.

\begin{proposition}
The function $\phi$ given by (\ref{functiondichedralpqr}) is positive definite on $W=\B D_{\infty}$
if and only if $0\le r$ and $|p-r|+|q-r|\le1-r$, i.e.
\begin{equation}\label{dichedralinfinite}
\max\left\{0,p+q-1\right\}\le r\le\min\left\{1-|p-q|,\frac{1+p+q}{3}\right\}.
\end{equation}
\end{proposition}

\begin{proof}
First we note that the set of $(p,q,r)\in\B R^3$ satisfying (\ref{dichedralinfinite})
constitutes a pyramid which is the convex hull of the points $(\pm1,0,0)$, $(0,\pm1,0)$ and $(1,1,1)$ (apex).
For these particular parameters it is easy to see that $\phi$ is positive definite: $(1,1,1)$ corresponds
to the constant function $1$, $(1,0,0)$ to the characteristic function of the subgroup $\left\langle s\right\rangle=\{1,s\}$,
and $(-1,0,0)$ to the character $\chi_{-,-}$ times the characteristic function of $\left\langle s\right\rangle$.
Similarly for $(0,\pm1,0)$. This, by convexity, proves that (\ref{dichedralinfinite}) is a sufficient condition.

On the other hand, we know already that $r\ge0$ is a necessary condition.
Let us fix $n$ and define $W^+(n)=\{x\in W:|sx|<|x|\le2n\}$, $W^-(n)=\{x\in W:|tx|<|x|\le2n\}$ and
\[
f(x)=
\begin{cases}
\pm1&\text{if $x\in W^\pm(n)$,}\\
0&\text{otherwise.}
\end{cases}
\]
For $x,y\in W^+(n)$ we have $S_{y^{-1}x}=\varnothing$ in $2n$ cases (namely, if $x=y$)
$S_{y^{-1}x}=\{s\}$ in $2n-2$ cases (namely if $|x|=2k$, $|y|=2k+1$ or vice-versa, $k=1,\ldots,n-1$)
$S_{y^{-1}x}=\{t\}$ in $2n$ cases (namely if $|x|=2k$, $|y|=2k-1$ or vice-versa, $k=1,\ldots,n$)
and $S_{y^{-1}x}=\{s,t\}$ in all the other $(2n-1)(2n-2)$ cases.
Similarly, for $x,y\in W^-(n)$ we have $S_{y^{-1}x}=\emptyset$ in $2n$ cases,
$S_{y^{-1}x}=\{s\}$ in $2n$ cases,
$S_{y^{-1}x}=\{t\}$ in $2n-2$ cases
and $S_{y^{-1}x}=\{s,t\}$ in $(2n-1)(2n-2)$ cases.
If $x\in W^+(n)$, $y\in W^-(n)$ or vice-versa then $S_{y^{-1}x}=\{s,t\}$.
Summing up, we get
\begin{align*}
\sum_{x,y\in W}&\phi(y^{-1}x)f_n(x)f_n(y)\\
&=4n+(4n-2)p+(4n-2)q+(4n-2)(2n-2)r-8n^2 r\\
&=4n+(4n-2)p+(4n-2)q-(12n-4)r.
\end{align*}
Therefore for every $n\in\B N$ we have a necessary condition
\[
1+\left(1-\frac{1}{2n}\right)p+\left(1-\frac{1}{2n}\right)q-\left(3-\frac{1}{n}\right)r\ge0.
\]
Letting $n\to\infty$ we get $1+p+q\ge3r$.

Put $x_k=stst\ldots$, $|x_k|=k$. Fix $n$ and define
\[
g(x)=
\begin{cases}
\chi_{-,+}(x)&\text{if $x=x_k$ for $1\le k\le 4n$,}\\
0&\text{otherwise,}
\end{cases}
\]
where, as before, $\chi_{-,+}$ is the character on $W$ for which $\chi_{-,+}(s)=-1$, $\chi_{-,+}(t)=1$.
Then
\[
\sum_{x,y\in W}\phi(y^{-1}x)g(x)g(y)=\sum_{k,l=1}^{4n}\phi(x_{l}^{-1}x_{k})g(x_k)g(x_l).
\]
Denote $c_{k,l}=\phi(x_{l}^{-1}x_{k})g(x_k)g(x_l)$. Then we have $c_{k,k}=1$, $1\le k\le 4n$,
$c_{k,k-1}=q$ if $k$ is even, $c_{k,k-1}=-p$ if $k$ is odd, $2\le k\le4n$ and
$c_{k,l}=c_{l,k}$ for all $1\le k,l\le 4n$. If $1\le k,l\le 4n$ and $|k-l|\ge2$
then $c_{k,l}=(-1)^{j}r$, where $j$ is the total number of $s$ appearing in $x_k$ and $x_l$.
Now it is not difficult to check that
\[
\sum_{l=1}^{4n}c_{k,l}=
\begin{cases}
1+q-2r&\text{if $k=1$ or $k=4n$,}\\
1-p+q-r&\text{if $1<k<4n$,}
\end{cases}
\]
which implies
\[
\sum_{x,y\in W}\phi(y^{-1}x)g(x)g(y)=4n-(4n-2)p+4nq-(4n+2)r
\]
and leads to necessary condition $r\le 1-p+q$.
In a similar manner we get $r\le1+p-q$.

Finally, define a function $h$ similarly like $g$, but now we use the character $\chi_{-,-}$:
\[
h(x)=
\begin{cases}
\chi_{-,-}(x)=(-1)^k&\text{if $x=x_k$ for $1\le k\le 4n$,}\\
0&\text{otherwise.}
\end{cases}
\]
Putting $d_{k,l}=\phi(x_{l}^{-1}x_{k})h(x_k)h(x_l)$ we have
$d_{k,k}=1$, $d_{k,k-1}=-p$ if $2\le k\le 4n$ is even
and $d_{k,k-1}=-q$ if $k$ is odd. Moreover, if $|k-l|\ge2$, $1\le k,l\le 4n$
then $d_{k,l}=(-1)^{k+l} r$. Now one can check that
\[
\sum_{l=1}^{4n}d_{k,l}=
\begin{cases}
1-q&\text{if $k=1$ or $k=4n$,}\\
1-p-q+r&\text{if $1<k<4n$,}
\end{cases}
\]
which yields $1-p-q+r\ge0$
and completes the proof that the conditions (\ref{dichedralinfinite}) are necessary.
\end{proof}

\section{Weak Sidon sets and operator Khinchin inequality}

The aim of this section is to show that the set of Coxeter generators $S$ in an arbitrary
Coxeter group $W$ is a \emph{weak Sidon set}, ie. an interpolation set for the Fourier--Stieltjes
algebra $B(W)$.

%


Given a group $\Gamma$, the Fourier--Stieltjes algebra consists of linear combinations
of positive definite functions on $\Gamma$, ie. every element of $B(\Gamma)$ is of the form
$f=\varphi_1-\varphi_2+i(\varphi_3-\varphi_4)$ for some  positive definite functions $\varphi_i$
($1\leq i\leq4$) on $\Gamma$.  The norm on $B(\Gamma)$ is defined as
$$
\|f\|_{B(\Gamma)}=\inf\left\{\sum \varphi_i(e)|\hbox{\rm where $f$ decomposes as above}\right\}
$$

\begin{theorem}
\label{T:5.1}
The set of Coxeter generators $S$ in an arbitrary
Coxeter group $W$ is a weak Sidon set, ie. for every bounded function $f\colon S\to[-1,1]$
there exists positive definite functions $\varphi_\pm$, such that $f(s)=\varphi_+(s)-\varphi_-(s)$ for any $s\in S$.
One can take $\varphi_\pm=R_{\B q^\pm}$ for a suitable choice of\/ $\B q^\pm$.  Moreover
$$
\|\varphi_+-\varphi_-\|_{B(W)}\leq 2
$$
\end{theorem}

\begin{proof}
Put $S_\pm(f)=\{s\in S|\pm f(s)> 0\}$.  Set
$$
q^\pm_s=\begin{cases}\pm f(s)&\hbox{for $s\in S_\pm(f)$,}\\0&\hbox{otherwise.}\end{cases}
$$
Then $f(s)=R_{\B q^+}(s)-R_{\B q^-}(s)$ as claimed.  The rest of the statement hold as
the Riesz-Coxeter function at the identity element equals to one.
\end{proof}

Given a matrix $A\in M_n(\B C)$ and $p\geq 1$ the Schatten $p$-class norm $\|A\|_{\C S_p}$
is defined as $\|A\|^p_{\C S_p}=\left(\tr|A|\right)^{1/p}$, where $|A|=(A^*A)^{1/2}$.

Let $\lambda$ denote the left regular representation of a group $\Gamma$.  Given a finite sum
$f=\sum c_g\lambda(g)\in\B C[\Gamma]$ we define noncommutative $L^p$-norm
$$
\|f\|_{L^p(\Gamma)}^p=\left(\tau\left((f^**f)^{1/2}\right)\right)^{1/p}
$$
where $\tau(f)=c_e$ is the von Neumann trace and $L^p(\Gamma)$ is a completion
of $\B C[\Gamma]$ with respect to the above norm.

We recall, that a scalar-valued map $\varphi$ on a group $\Gamma$ is called a \emph{completely bounded Fourier multiplier}
on $L^p(\Gamma)$ if the associated operator
$$
M_\varphi(\lambda(g))=\varphi(g)\lambda(g),\qquad g\in \Gamma
$$
extends to a completely bounded operator on $L^p(\Gamma)$.

We let $M_\cb(L^p(\Gamma))$ to be an algebra of completely bounded Fourier multipliers equipped
with the norm
$$
\|\varphi\|_{M_\cb(L^p(\Gamma))}=\|M_\varphi\otimes\mathop{\rm id}\nolimits_{\C S^p}\|.
$$

\def\gargantul{\left\|(a_s)_{s\in S}\right\|_{\mathop{R}\cap \mathop{C}}}

Following Pisier \cite{MR2006539}, for $a_s\in M_n(\B C)$, where $s\in S$, we define
$$
\gargantul={\max\left\{\left\|\sum_{s\in S}(a_sa_s^*)^{1/2}\right\|_{\C S^p},\left\|\sum_{s\in S}(a_s^*a_s)^{1/2}\right\|_{\C S^p}\right\}}.
$$ 
For a set $E\in \Gamma$ we define the \emph{completely bounded} constant $\Lambda^\cb_p(E)$ as infimum of $C$ 
such that
$$
\left\|\sum_{s\in S}a_s\otimes\lambda(s)\right\|_{L^p(W)}\leq C\gargantul
$$
for all matrices $a_s\in M_n(\B C)$ and $n\in\B N$.

\begin{theorem}
\label{T:6.1}
If\/ $a_s\in M_n(\B C)$, then for all $p\geq 2$ and any Coxeter system $(W,S)$
we have
$$
\gargantul\leq\left\|\sum_{s\in S}a_s\otimes\lambda(s)\right\|_{L^p(W)}\leq 2A'\sqrt p\gargantul.
$$
\end{theorem}

\begin{proof}
It was shown by Harcharras \cite[Prop.~1.8]{MR859804} that $\Lambda^\cb_p(E)$ if finite
if and only if $E$ is an interpolation set for $M_\cb(L^p(\Gamma))$, i.e.
every bounded function on $E$ can be extended to a multiplier, and
$$
\Lambda^\cb_p(E)\leq \Lambda^\cb_p(R)\mu^\cb_p(E),
$$
where $R$ is the generating set in the Rademacher group $\Rad_\infty$
and $\mu^\cb_p(E)$ is the interpolation constant.

As shown by Buchholz \cite[Thm.~5]{MR2213610} for $p=2n$, and $S$ te standart generating set in $\Rad_\infty$,
$\Lambda_{2n}^\cb(R)=\left((2n-1)!!\right)^{1/2n}\leq A\sqrt p$ for some absolute $A$.  This was extended
by Pisier \cite[Thm.~9.8.2]{MR2006539} for any $p\geq 2$, i.e
$$
\Lambda_p^\cb(R)\leq A'\sqrt p,
$$
for an absolute constant $A'$.

We have shown in Theorem \ref{T:5.1} that in an arbitrary Coxeter group $W$ its Coxeter generating set $S$ is
a weak Sidon set, i.e. it is interpolation set for the Fourier--Stieltjes algebra $B(W)$.
Since for $p\geq 1$,
$B(\Gamma)$ is a subalgebra of $M_\cb(L^p(\Gamma))$ and
$$
\|\varphi\|_{M_\cb(L^p(\Gamma))}\leq\|\varphi\|_{B(\Gamma)},
$$
we see that $\mu^\cb_p(S)\leq 2$.  Thus $\Lambda^\cb_p(S)\leq 2A'\sqrt p$.
This finishes the proof of the right inequality.

The left inequality
holds for any group $\Gamma$ and any $S\subset \Gamma$ (see \cite{MR859804}).
\end{proof}

\begin{remark}
Fendler \cite{MR1967380} has shown  that if for all $s,t\in S$, $s\neq t$, we have $m_{s,t}\geq 3$, then
$$
\Lambda_p^\cb(S)\leq 2\sqrt 2.
$$
See also \cite{MR0390658} and \cite{MR1476122} for related results in the case of free Coxeter groups.
Also Haagerup and Pisier have shown that $\Lambda_\infty^\cb(S)=2$,
where $\Lambda^\cb_\infty(E)=\sup_{p\geq 2}\Lambda^\cb_p(E)$ \cite{MR1240608}.
See the paper of Haagerup \cite{MR654838} where the best constant
was calculated for the set of Coxeter generators of the Rademacher group in case when $a_s$ are scalars.
\end{remark}

\section{Chromatic length function for Coxeter groups and pairpartitions}

\def\NC{\mathop{\C{NC}}\nolimits}
\def\P{\mathop{\C P}\nolimits}
\def\r#1{{:}\mkern-.5\thinmuskip{#1}\mkern-.5\thinmuskip{:}}

Let $[2n]=\{1,\dots,2n\}$.  Let $2^{[2n]}$ denote the set of subsets of $[2n]$.
By a partition of $[2n]$  we mean $\pi\subset s^{[2n]}$ such that
$\bigcup\pi=[2n]$ and if $\pi',\pi''\in \pi$ then $\pi'=\pi''$ or
$\pi'\cap\pi''=\varnothing$.  We say, that partition $\varrho$ is a \emph{coarsening}
of a partition $\pi$ if for any $\pi'\in\pi$ there exists $\varrho'\in\varrho$
such that $\pi'\subset\varrho'$.

A partition is called \emph{crossing}
if there exist $1\leq a<b<c<d\leq 2n$ and $\pi_1,\pi_2\in\pi$ with
$a,c\in\pi_1\neq\pi_2\ni b,d$; otherwise it is called \emph{noncrossing}.
For any partition $\pi$ there exists th the smallest noncrossing coarsening $\Phi(\pi)$ of $\pi$
(ie.\ if $\varrho$ is a noncrosing coarsening of $\pi$ then it is a coarsening of $\Phi(\pi)$).
We define $\|\pi\|=n-\#\Phi(\pi)$.  The notion for the map $\Phi$ was introduced in \cite{MR2423121}.

We say that $\pi$ is a \emph{pairpartition} if every member of $\pi$ has cardinality two.
The set of pairpartitions of $[2n]$ is denoted by $\P_2(2n)$.
Given $\pi\in\P_2(2n)$ we write $|\pi|$ to denote the number of ordered quadruples
$1\leq a<b<c<d\leq 2n$ such that both $\{a,c\}$ and $\{b,d\}$ belong to $\pi$.
Note, that $|\pi|=0$ precisely when $\pi$ is noncrossing.  The set of
noncrossing pairpartitions is denoted $\NC_2(2n)$.

\def\inner{\mathop{\hbox{\rm inn}}\nolimits}
Given a noncrossing pairpartition $\varpi$ we call $\{b,c\}\in\varpi$ \emph{an inner block} if there exists $\{a,d\}\in\varpi$
with $a<b<c<d$.  The number of inner blocks of $\varpi$ we denote as $\inner(\varpi)$.

In \cite[Cor.~7]{MR1388006} Bo\.zejko and Speicher observed the following identity.
\begin{equation}
\label{eq:p2-nc2}
\sum_{\pi\in\P_2(2n)}(-1)^{|\pi|}q^{\|\pi\|}=\sum_{\varpi\in\NC_2(2n)}(1-q)^{\inner(\varpi)}.
\end{equation}

Let $f_n(q)=\sum_{\varpi\in\NC_2(2n)}(1-q)^{\inner(\varpi)}$.  It is
elementary to derive
$$
f_n(q)=C_n\,\,\mathstrut_2\!F_1\left({n,1-n\atop n+2}\right|\left.\vphantom{n,1-n\atop n+2}q\right),
$$
where $C_n={2n\choose n}-{2n\choose n-1}=\#\NC_2(2n)$ denote the $n$-th Catalan number
and $\mathstrut_2\!F_1$ is the classical hypergeometric funcion.
If we write $f(q)=\sum_{j=0}^{n-1}t^n_jq^j$, then the triangle $(t^n_j)_{0\leq j<n}$
appears in \cite{oeisA062991} as ``sequence''
\textsc{\MakeLowercase{A062991}}).  Since we are not going to use this formula, we leave it as an exercise to the reader.
For the expansion of $f_n(1-t)$ and the Delanoy triangle appearing there the reader may consult \cite[Prop~6.1]{MR1863276}.


In what follows, we prove a result about an arbitrary finitely generated Coxeter group, which for permutation
groups implies the above one.
Given a permutation $\sigma\in S_n$, we construct a pairpartition $\overline\sigma=\{\{i,2n+1-\sigma(i)\}|1\leq i\leq n\}$.
Note, that $|\overline\sigma|$ is equal to the length of $\sigma$ with respect to the Coxeter generators $(1,2)$,
\dots, $(n-1,n)$ of $S_n$.  Therefore, denoting by $|w|$
the Coxeter length of an element $w$ of some Coxeter group $W$ will not lead into any confusion.

It is also clear that, with respect to the identification of permutations with
a subset of pairpartitions, the two definitions of $\|\cdot\|$ agree
(see Equations (\ref{def:block-length}) and (\ref{def:colour})).

By $W(t)$ we denote a growth series of a finitely generated Coxeter group $W$
(length function).  That is, a power series $W(t)=\sum_{w\in W}
t^{|w|}$.  (Note, that the coefficient at $t$ equals to $\#S$.  This
explains why here and in the rest of this section we consider only finitely
generated Coxeter groups.  We will not repeat this assumption for short.)
Moreover, for $X\subset W$ we write
$X(t)=\sum_{w\in X} t^{|w|}$.


Let us define a multivariable formal power series (\emph{chromatic length function}.
For any $X\subset W$ define
$$
X(t,\B q)=\sum_{w\in X}t^{|w|}\prod_{s\in S_w}q_s.
$$
In particular $X(t)=X(t,\B 1)$, where $\B 1=(1)_{s\in S}$.

\begin{proposition}\label{P:tokyo}
The polynomial (or formal power series, if\/ $W$ is infinite) $W(t,\B q)$ satisfies
$$W(t,\B q)=
\sum_{T\subset S}W_T(t)\prod_{r\in T}r_s\prod_{s\in S\smallsetminus T}(1-q_s).$$
\end{proposition}

\begin{proof}
Let\/ $W^\circ_{R}$ denote the set of all elements of\/ $W_R$ not
contained in any proper parabolic subgroup of\/ $W_R$, ie.
$W^\circ_{R}\colon=W_{R}-\bigcup_{T\varsubsetneq R}W_T$.
Then, by inclusion-exclusion principle, $W^\circ_R(t)=\sum_{T\subset R}(-1)^{\#(R-T)}W_T(t)$.
Therefore,
\begin{align*}
W(t,\B q)&=\sum_{w\in W}t^{|w|}\prod_{s\in S_w}q_s
=\sum_{R\subset S}W^\circ_{R}(t)\prod_{r\in R}q_r\\
&=\sum_{R\subset S}
\sum_{T\subset R}W_T(t)(-1)^{\#(R-T)}\prod_{r\in R}q_r\\
&=\sum_{T\subset S}
W_T(t)\prod_{r\in T}q_r\sum_{T\subset R\subset S}\prod_{s\in R\smallsetminus T}(-q_s)\\
&=\sum_{T\subset S}W_T(t)\prod_{r\in T}q_i\prod_{s\in S\smallsetminus T}(1-q_s).
\end{align*}
\end{proof}

\begin{corollary}\label{C:wat-1}
If\/ $W$ is a finite Coxeter group then
\begin{equation}
\label{eq:wat-1}
W(-1,\B q)=\prod\limits_{s\in S}(1-q_s).
\end{equation}
\end{corollary}

\begin{proof}
Choose $s\in T$ and put $W_T^{\{s\}}=\{w\in W:|w|<|ws|\}$.  Clearly, $W_T=W_T^{\{s\}}W_T$
therefore $W_T(t)=W_T^{\{s\}}(t)W_{\{s\}}(t)$.
Since $W$ is a finite group, $W_T^{\{s\}}$ is a polynomial.
Thus $W_T(-1)=0$ if $T$ is nonempty (and $W_\varnothing(-1)=1$).
\end{proof}

In order to prove Equation (\ref{eq:p2-nc2}) we define the \emph{Wick map} $\P_2(2n)\ni\pi\mapsto\r{\pi}\in\NC_2(2n)$
(related to the normal order in quantum field theory).
Given a pairpartition $\pi$ we define $\r{\pi}$ by repetitive resolving crossings.
That is, we replace repetitively every crossing pair $\{a,c\}$ and $\{b,d\}$ with $a<b<c<d$ by
$\{a,d\}$ and $\{b,c\}$.  In order to see that the result is independent of the
order of resolution we describe $\r{\pi}$ in an equivalent way.

Let $\Phi(\pi)$ be the smallest noncrossing coarsening of $\pi$.  For each block $\beta$
of $\Phi(\pi)$ define $\beta^+=\{y|(\exists x)\ x\in\beta,\ y>x,\ \{x,y\}\in\pi\}$ and
$\beta_1^-=\{x|(\exists y)\ y\in\beta,\ y>x,\ \{x,y\}\in\pi\}$.  Order $\beta^+=\{y_1,\dots,y_k\}$
in increasing way and $\beta^-=\{x_1,\dots,x_k\}$ in decreasing way.  Then all
pairs $\{x_i,y_i\}$ will be parts of $\r{\pi}$.

Equation (\ref{eq:p2-nc2}) will follow from a more refined statement.

\begin{proposition}
\label{P:tokyo2}
For every $\varpi\in\NC_2(2n)$
\begin{equation}
\label{eq:tokyo2}
\sum_{\pi\in\P_2(2n)\atop \r{\pi}=\varpi}(-1)^{|\pi|}q^{\|\pi\|}=(1-q)^{\inner(\varpi)}.
\end{equation}
\end{proposition}

\begin{proof}
Let us first consider the case of $\varpi=\overline1=\{(i,2n+1-i)|1\leq i\leq n\}$.
Clearly, $\{\pi|\r{\pi}=\overline1\}=\{\overline\sigma|\sigma\in\F S_n\}$.  And Equation
(\ref{eq:tokyo2}) is equivalent to Equation (\ref{eq:wat-1}) (as all $q_s$ are set to $q$) for
$W=\F S_n$.

In a general case observe, that $\Phi(\pi)$ is a coarsening of $\varpi=\r{\pi}$.  Yet, not every coarsening
may appear.  The obvious condition is that for each block $\beta$ of $\r{\pi}$ the pair $\{\min\beta,\max\beta\}$
belong to $\varpi$.  For the purpose of this proof we will call such a coarsening
admissible.  Its clear, that abmissible coarsenings $\rho$ are in one to one correspondence with
subsets of $\varpi$ containing all outer (not inner) parts of $\varpi$ of the form $\{\{\min\rho',\max\rho'\}|\rho'\in\rho\}$.

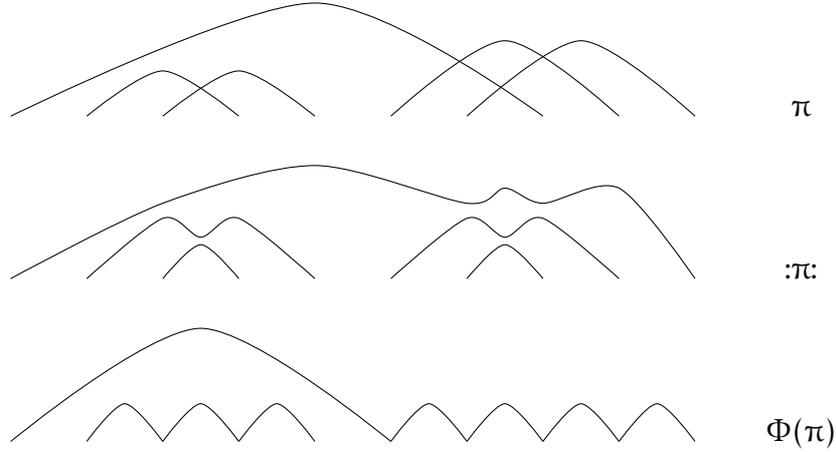
\begin{figure}
\centering
\begin{tikzpicture}
\matrix[matrix of nodes,column sep=0.8cm,row sep=0.5cm]{
\draw plot [smooth] coordinates {(1,0) (5,1.5) (8,0)};
\draw plot [smooth] coordinates {(2,0) (3,0.6) (4,0)};
\draw plot [smooth] coordinates {(3,0) (4,0.6) (5,0)};
\draw plot [smooth] coordinates {(6,0) (7.5,1) (9,0)};
\draw plot [smooth] coordinates {(7,0) (8.5,1.0) (10,0)};
&$\pi$\\
\draw plot [smooth] coordinates {(1,0) (3,1) (5,1.5) (7,1) (7.5,1.2) (8,1) (9,1.2) (10,0)};
\draw plot [smooth] coordinates {(2,0) (3,.8) (3.5,.55) (4,.8) (5,0)};
\draw plot [smooth] coordinates { (3,0) (3.5,0.45) (4,0)};
\draw plot [smooth] coordinates { (6,0) (7,0.8) (7.5,0.55) (8,0.8) (9,0)};
\draw plot [smooth] coordinates { (7,0) (7.5,.45) (8,0)};
&$\r{\pi}$\\
\draw plot [smooth] coordinates {(1,0) (3.5,1.5) (6,0)};
\draw plot [smooth] coordinates {(6,0) (6.5,.5) (7,0)};
\draw plot [smooth] coordinates {(7,0) (7.5,.5) (8,0)};
\draw plot [smooth] coordinates {(8,0) (8.5,.5) (9,0)};
\draw plot [smooth] coordinates {(9,0) (9.5,.5) (10,0)};
\draw plot [smooth] coordinates {(2,0) (2.5,.5) (3,0)};
\draw plot [smooth] coordinates {(3,0) (3.5,.5) (4,0)};
\draw plot [smooth] coordinates {(4,0) (4.5,.5) (5,0)};
&$\Phi(\pi)$\\};
\end{tikzpicture}
\caption*{\textsc{Figure.} Examples of $\pi$, $\r{\pi}$, and $\Phi(\pi)$.}
\end{figure}

Let us refine Equation (\ref{eq:tokyo2}) further.  For every $\varpi\in\NC_2(2n)$ and any
admissible coarsening $\eta$ of $\varpi$ we have
\begin{equation}
\label{eq:tokyo3}
\sum_{\pi\in\P_2(2n)\atop \r{\pi}=\varpi,\ \Phi(\pi)=\rho}(-1)^{|\pi|}=(-1)^{\#\rho}.
\end{equation}
Equation (\ref{eq:tokyo2}) follows from (\ref{eq:tokyo3}) by multiplying by $q^{n-\#\eta}$
and summing over all admissible coarsenings $\eta$ of $\varpi$.

Equation (\ref{eq:tokyo3}) is again equivalent to to Equation (\ref{eq:wat-1})
(for all apermutation groups and all $q_s$ set to $q$) as both
sides factor as a product over blocks of $\rho$.
\end{proof}

\begin{question}
We have proven Equation \ref{eq:p2-nc2} with the help of an embedding $\F S_n\ni\sigma\mapsto\overline\sigma\in\NC_2(2n)$
(or several such embeddings, one for each outer block of\/ $\varpi$).  Corollary
\ref{C:wat-1} holds for any Coxeter group.  Is there a corresponding formula
concerning some generalization of pairpartitions?
\end{question}

In the proof of Proposition \ref{P:tokyo} we have not assumed that $W$ was finite.
Let us finish this section with a discussion of infinite Coxeter groups.  Recall,
that $-1$ does not lie in the radius of convergence on $W(t)$ if $W$ is not finite.
Nevertheless, $W(t)$ represents a rational function as follows from the following result.
\begin{proposition}(\cite{MR0230728},\cite[Prop.~26]{MR0385006})\label{P:Serre}
Let $(W,S)$ be an an infinite Coxeter system. Then
\begin{equation}
\label{eq:serre}
\frac{1}{W(t)}=\sum_{T\in\C F}\frac{(-1)^{\#T}}{W_T(1/t)}.
\end{equation}
Where $\C F$ denote the family of subsets $T\subset S$, such that
the group $W_T$ generated by $T$ is finite.  In particular, $W(t)$ is a series
of a rational function (i.e. a quotient of polynomials).
\end{proposition}

One may ask a question what is the class of (infinite) Coxeter groups such that
$W_T(-1)=0$ for any nonempty subset $T$ of generators.
A~na\"ive argument that
$$W(t)=W_{\{s\}}(t)W^{\{s\}}(t)=(1+t)W^{\{s\}}(t)$$
shows, that the question if $W(-1)\neq0$ is equivalent to whether
$W^{\{s\}}(t)$ can have a pole at $t=-1$.
On the other hand note, that if $W$ is of type \textsc{\MakeLowercase{\~A}}$\mathstrut_2$, ie.~$W$
is given by a presentation $\langle s_i:1\leq i\leq3|s_i^2,(s_isj)^3:1\leq i<j\leq3\rangle$
then, by Equation (\ref{eq:serre}),  $W(t)=\frac{1+t+t^2}{(1-t)^2}$ and $W(-1)=\sfrac14$.

More generally, it is known (\cite{MR0240238}) that in each coset of $W_T$
there exists the unique shortest element.
Let $W^T$ denote the set of those shortest representatives.
Moreover if $w=w^Tw_T$ with $w^T\in W^T$ and $w_T\in W_T$ then
$|w|=|w^T|+|w_T|$. Therefore $W^T(t)W_T(t)=W(t).$
In particular, $W^T(t)$ represents a rational function,
and it is legitimate to ask about the value of $W^T(-1)$.

In the case of finite Coxeter group $W$, Eng \cite{MR1817702} observed that
$$W^T(-1)=\#\left\{w\in W^T|ww_0w\in W_T\right\},$$ where $w_0$ is the longest
element in $W$.
(Eng's proof was case-by-case. Later, a  general classification-free proof of Eng's theorem
was given in \cite{MR2087303}).

Subsequently, Reiner \cite{MR1923093} has shown that if $W$ is crystallographic
(ie. the Weyl group in a compact Lie group $G$),
then both sides of the above equality compute the signature of the corresponding
flag variety $G/Q_T$, where $Q_T$ is a parabolic subgroup associated to $T$.

What is the meaning of $W(-1)$ or $W^T(-1)$ for infinite $W$?

We do not know if it possible for $W(-1)$ to be negative.
If one takes $W=\langle s_i:1\leq i\leq 4|s_i^2, (s_is_j)^3:1\leq i<j\leq4\rangle$.
Then, by Equation (\ref{eq:serre}),  $W(t)=\frac{(1+t)(1+t+t^2)}{3t^3-2t^2-2t+1}$ and
$W^{s_i}(-1)=\sfrac{-1}{2}$.

\bibliography{tokyo}
\bibliographystyle{sc}

\section*{Acknowledgemens}

Marek Bo{\.z}ejko was supported by \textsc{\MakeLowercase{NCN MAESTRO}} grant \textsc{\MakeLowercase{DEC-2011/02/\discretionary{}{}{}A/ST1/00119}}.
Marek Bo{\.z}ejko and Wojciech M{\l}otkowski were supported by \textsc{\MakeLowercase{NCN}} grant \textsc{\MakeLowercase{2016/21/B/ST1/00628}}.

The authors would like to thank Ryszard Szwarc and Janusz Wysoczański for many discussions and help
with the preparation of this paper.

The first author would like to thank professor
F.~G\"otze for his kind invitation to \textsc{\MakeLowercase{SFB701}} and
his hospitality in Bielefeld in 2016.
\end{document}